\documentclass[12pt]{amsart}

\usepackage{bm,amsfonts,amsmath,amssymb,dsfont,mathrsfs,bbm} 
\usepackage{mathptmx} 

\makeatletter
\def\section{\@startsection{section}{1}\z@{.9\linespacing\@plus\linespacing}%
  {.7\linespacing} {\fontsize{13}{14}\selectfont\bfseries\centering}}
\def\paragraph{\@startsection{paragraph}{4}%
  \z@{0.3em}{-.5em}%
  {$\bullet$ \ \normalfont\itshape}}

\@addtoreset{equation}{section}
\makeatother

\newtheorem{theorem}{Theorem}[section]

\newtheorem{lemma}[theorem]{Lemma}
\newtheorem{corollary}[theorem]{Corollary}

\theoremstyle{definition}
\newtheorem{definition}[theorem]{Definition}

\theoremstyle{remark}
\newtheorem{remark}[theorem]{Remark}

\usepackage{color}	

\definecolor{webred}{rgb}{0.75,0,0}
\definecolor{webgreen}{rgb}{0,0.75,0}
\usepackage[citecolor=webgreen,colorlinks=true,linkcolor=webred]{hyperref}

\newcommand{\ba}{\mathbf{a}}
\newcommand{\cB}{\mathcal{B}}
\newcommand{\cBp}{\mathcal{B}'}

\newcommand{\bc}{\mathbf{c}}
\newcommand{\bbC}{\mathbb{C}}
\newcommand{\cC}{\mathcal{C}}
\newcommand{\dd}{\mathrm{d}}
\newcommand{\cD}{\mathcal{D}}
\newcommand{\bE}{\mathbf{E}}
\newcommand{\bF}{\mathbf{F}}
\newcommand{\be}{\mathbf{e}}
\newcommand{\rme}{\mathrm{e}}
\newcommand{\boldf}{\mathbf{f}}

\newcommand{\bH}{\mathbf{H}}
\newcommand{\bbH}{\mathbb{H}}
\newcommand{\rH}{\mathrm{H}}
\newcommand{\bh}{\mathbf{h}}
\newcommand{\bk}{\mathbf{k}}
\newcommand{\cKt}{\mathcal{K}}
\newcommand{\cKp}{\mathcal{K'}}
\newcommand{\cK}{\mathcal{K}}
\newcommand{\bbL}{\mathbb{L}}
\newcommand{\cL}{\mathcal{L}}
\newcommand{\rL}{\mathrm{L}}
\newcommand{\bbN}{\mathbb{N}}

\newcommand{\cP}{\mathcal{P}}
\newcommand{\rP}{\mathrm{P}}
\newcommand{\bbR}{\mathbb{R}}
\newcommand{\bu}{\mathbf{u}}
\newcommand{\bU}{\mathbf{U}}
\newcommand{\bv}{\mathbf{v}}
\newcommand{\bV}{\mathbf{V}}
\newcommand{\bx}{\mathbf{x}}
\newcommand{\bxi}{\bm{\xi}}
\newcommand{\bxip}{\bm{\xi}'}
\newcommand{\xib}{\,'_{\!\!(\bxi'\!,\beta)}}

\newcommand{\bbZ}{\mathbb{Z}}

\newcommand{\loc}{\mathrm{loc}}
\newcommand{\per}{\mathrm{per}}

\newcommand{\grad}{\operatorname{\mathbf{grad}}\,}
\newcommand{\Div}{\operatorname{\mathrm{div}}}
\newcommand{\curl}{\operatorname{\mathrm{curl}}}
\newcommand{\Curl}{\operatorname{\mathbf{curl}}\,}
\renewcommand{\Re}{\operatorname{\mathrm{Re}}}

\newcommand{\nrm}  [1] {\Vert #1 \Vert}
\newcommand{\tsfrac}  [2] { {\textstyle \frac{#1}{#2} } }
\newcommand{\beq}{\begin{equation}}
\newcommand{\eeq}{\end{equation}}

\title{Regularity for Maxwell eigenproblems in photonic crystal fibre modelling}
\author{Monique Dauge, Richard A. Norton, and Robert Scheichl}

\address{Monique Dauge,
           IRMAR, Universit\'{e} de Rennes 1, Campus de Beaulieu, 35042, Rennes Cedex France}
\email{monique.dauge@univ-rennes1.fr}
           
\address{Richard A. Norton,
           Department of Physics, University of Otago, PO Box 56, Dunedin 9054, New Zealand
           Tel.: +64-3-4797749 
           Fax: +64-3-4790964 }
           \email{richard.norton@otago.ac.nz}
           
\address{Robert Scheichl,
           Department of Mathematical Sciences, University of Bath, Bath BA2 7AY, UK }
		   \email{R.Scheichl@bath.ac.uk}



\begin{document}

\begin{abstract}
The convergence behaviour and the design of numerical methods 
for modelling the flow of light in photonic crystal fibres 
depend critically on an understanding of the regularity of 
solutions to time-harmonic Maxwell equations in a 
three-dimensional, periodic, translationally invariant, 
heterogeneous medium.  In this paper we determine the strength 
of the dominant singularities that occur at the interface 
between materials.  By modifying earlier regularity theory 
for polygonal interfaces we find that on each subdomain, where 
the material in the fibre is constant, the regularity of 
in-plane components of the magnetic field are $H^{2-\eta}$ for 
all $\eta > 0$. This estimate is sharp in the sense that these 
components do not belong to $H^2$, in general. However, global 
regularity is restricted by the presence of an interface 
between these subdomains and the interface conditions imply 
only $H^{3/2-\eta}$ regularity across the interface.  The 
results are useful to anyone applying a numerical method such 
as a finite element method or a planewave expansion method to 
model photonic crystal fibres or similar materials.
\end{abstract}

\maketitle

\section{Introduction}
\label{intro}

This paper is concerned with source-free time-harmonic Maxwell equations in a three-dimensional, periodic, heterogeneous medium that is non-magnetic.  The problem is: Find non-zero $(\bE,\bH) \in \bbL^2_{\loc}(\bbR^3) \times \bbL^2_{\loc}(\bbR^3)$ and $\omega \in \bbR$ such that
\begin{subequations}
\label{maxwell012}
\begin{align}
\label{maxwell01} \nabla \times \bE - i \omega \mu \bH &= 0, \\
\label{maxwell02} \nabla \times \bH + i \omega \epsilon \bE &= 0, 
\end{align}
\end{subequations}
where $\bbL^2_{\loc}(\bbR^3) = (\rL^2_{\loc}(\bbR^3))^3$ is the space of locally square integrable vector fields on $\bbR^3$, the magnetic permeability $\mu = \mu_0$ is constant and equal to the permeability of a vacuum, and the electric permittivity $\epsilon$ (also called the dielectric) is a given function which is positive, bounded and with bounded inverse.  The vector fields $\bE$ and $\bH$ are respectively the electric field and magnetic field, and $\omega$ is the frequency. 

An alternative equivalent formulation of problem \eqref{maxwell012} is the following eigenproblem where the electric field has been eliminated: Find non-zero $\bH \in \bbL^2_{\loc}(\bbR^3)$ and $\kappa \in \bbR$ such that
\begin{subequations}
\label{maxwell12}
\begin{align}
\label{maxwell1}	\nabla \times ( \tsfrac{1}{n^2} \nabla \times \bH ) &= \kappa^2 \bH, \\
\label{maxwell2}	\nabla \cdot \bH &= 0,
\end{align}
\end{subequations}
where $n$ is the refractive index of the material and is related to $\epsilon$ by $\epsilon = \epsilon_0 n^2$ (where constant $\epsilon_0$ is the permittivity of free space).  The wave number $\kappa$ is related to $\omega$ by $\kappa^2 = \epsilon_0 \mu_0 \omega^2$.  
As we will see in this paper, the motivation for using \eqref{maxwell12} instead of \eqref{maxwell012} is that the regularity of $\bH$ is better than $\bE$, so we expect numerical methods to perform better on \eqref{maxwell12}.

The electric permittivity $\epsilon$ is assumed periodic with respect to a given lattice $\cKt$: For linearly independent primitive lattice vectors $\ba_1, \ba_2 , \ba_3\in \bbR^3$, the lattice $\cKt$ is the set $\{ \bk \in \bbR^3 :\ \bk = k_1 \ba_1 + k_2 \ba_2,+ k_3 \ba_3, \ k_1,k_2,k_3 \in \bbZ \}$ and $\epsilon$ satisfies
\begin{equation}
\label{eq:epsgen}
	\epsilon(\bx + \bk) = \epsilon(\bx) \qquad 
	\mbox{for all \ $\bx \in \bbR^3$ \ and \ $\bk \in \cKt$}.
\end{equation}
Applying the Floquet-Bloch transform translates the problem on $\bbR^3$ into a family of problems on the periodicity cell (torus) $Q = \bbR^3 / \cKt$. The new problems are thus posed on a compact manifold, and the new operators have compact resolvent and discrete spectra.  The spectrum (in the form of spectral bands) of the original (un-transformed) problem is then obtained by taking the union of the spectra of the family of transformed problems (c.f. \cite{kuchmentbook,kuchmentarticle,figotin1997}). 
 
In this paper, we focus in particular on problems arising from the propagation of light in photonic crystal fibres (PCF), novel optical devices that overcome the limitations of conventional fibre optics \cite{knight,russell,joannopoulosbook}. In PCFs we have translational invariance along the length of the fibre, as well as periodicity in the transverse directions. Let $(x,y,z)$ be the coordinates of the generic point $\bx\in\bbR^3$ and assume that the medium is translationally invariant in the $z$-direction, i.e., $\epsilon = \epsilon(\bx')$ with $\bx'$ denoting the transverse variables $(x,y)$. In this case the periodicity is relative to a two-dimensional lattice $\cKp$ with primitive vectors $\ba_1$, $\ba_2$ and $\epsilon$ satisfies
\begin{equation}
\label{eq:epstrans}
	\epsilon(\bx' + \bk') = \epsilon(\bx') \qquad 
	\mbox{for all \ $\bx' \in \bbR^2$ \ and \ $\bk' \in \cKp$}.
\end{equation}
Then the three-dimensional Floquet-Bloch transform degenerates into a two-dimen\-sional Floquet-Bloch transform in transverse variables $\bx'$ and a partial Fourier transform in longitudinal variable $z$. Again we obtain a family of Maxwell problems posed on a compact manifold without boundary, the torus $Q' = \bbR^2 / \cKp$ (period cell of $\cKp$).

Besides proving basic regularity in $\rH^1$ for each component of the
magnetic field for both the primitive equations \eqref{maxwell12} and
the problems deduced by Floquet-Bloch transforms in general periodic media, 
the main aim of this
paper is to establish optimal regularity results in the case of PCFs
with polygonal cross section. In this latter situation, our results
are more precise than those provided by the general regularity theory
for Maxwell interface problems in \cite{costabel00,costabel}. Let us
mention also the paper \cite{elschner} which studies the same problem as
we do in the context of diffraction gratings, however considering
only the regularity of the $z$-components of $\bH$ and $\bE$. Here,
by focusing on the regularity of the $x$- and $y$-components of $\bH$
we manage to obtain better regularity results, even in the more
general situation of an interface between any finite number of materials 
with distinct dielectric values. The results carry over to piecewise 
$C^2$ cross sections without cusps.  The improved regularity results are
important for the design and analysis of numerical methods for PCF
modelling. It is important to note that the PCF problem that we
investigate is the full vectorial problem (without additional
simplifications), as studied in the physics literature
\cite{ferrandoetal,saitohkoshiba,pottage2003robust,nicoletetal,pearce2005adaptivecurvilinear,joannopoulosbook}. 
See \cite{norton2012apnum} for more details.

The paper is organised as follows: The rationale is to examine the
regularity of solutions of problem \eqref{maxwell12} under more and
more specific assumptions on the electric permittivity $\epsilon$.  In
\S\ref{sec:bounded} we simply suppose that the material is
non-magnetic and $\epsilon$ bounded with bounded inverse, and prove the
$\rH^1$--regularity for $\bH$. In \S\ref{sec:periodic} we impose
periodicity and exploit symmetries to derive a weak formulation for 
the family of
Floquet transformed operators on a compact manifold.  We prove the
existence of a sequence of real eigenvalues with corresponding
eigenvectors.  In \S\ref{sec:pertrans} we assume in addition that
$\epsilon$ is translationally invariant and obtain more specific
results. In the remaining sections we then focus on piecewise constant
permittivity $\epsilon$ on a polyhedral partition of the full space,
still assuming periodicity. First, in \S\ref{sec:reg} the assumptions
are again very general and we recall results from the literature. In
\S\ref{sec:reg+} we add the assumption of invariance in one direction
and prove the main result of the paper: piecewise
$\rH^{2-\eta}$--regularity for the transverse components of the
magnetic field, for any $\eta >0$. In \S\ref{sec:reg++} we finally
assume that in addition there are only two materials with simple edge 
interfaces (no cross points), which is typical for PCFs. We recall a 
result from \cite{elschner} which gives 
the regularity of the longitudinal components of the electric and
magnetic fields, and deduce an explicit expansion for the transverse 
components of the magnetic field which slightly improves the result
from \S\ref{sec:reg+}, but also shows that the eigenfunctions are 
not piecewise $\rH^2$. The final section \S\ref{sec fem} contains some
conclusions and a discussion of how our new result can be applied in the
convergence theory of Galerkin methods.

\section{Bounded electric permittivity}
\label{sec:bounded}

Revisiting classical arguments we prove the following theorem.
\begin{theorem}
\label{th:1}
Let the magnetic permeability $\mu=\mu_0$ be constant and let the electric permittivity satisfy
\[
   \epsilon>0,\quad \epsilon\in\rL^{\!\infty}_\loc(\bbR^3), 
   \quad \epsilon^{-1}\in\rL^{\!\infty}_\loc(\bbR^3).
\]
(i) If $(\bE,\bH)$ is a solution of \eqref{maxwell012} in
$\bbL_\loc^2(\bbR^3)\times\bbL^2_\loc(\bbR^3)$ with $\omega\neq0$,
then $\bH$ is a solution of \eqref{maxwell12} and belongs to
$\bbH^1_\loc(\bbR^3)$. 
(ii) On the other hand, if $\bH$ and $\nabla\times\bH$ belong to $\bbL^2_\loc(\bbR^3)$ and are such that equations \eqref{maxwell12} are satisfied, then $\bH$ belongs to $\bbH^1_\loc(\bbR^3)$ and, setting $\bE = \frac{i}{\omega\epsilon} \nabla \times \bH$, we find a solution $(\bE,\bH)$ of equations \eqref{maxwell012}.
\end{theorem}

\begin{proof}
Taking the divergence of \eqref{maxwell01} yields $\nabla\cdot \bH = 0$ since $\mu$ is constant and $\omega\neq0$. So we have \eqref{maxwell2}. Equation \eqref{maxwell02} implies 
\[
   \bE = \frac{i}{\omega\epsilon} \nabla \times \bH 
\] 
and \eqref{maxwell01} in the distributional sense reads
\[
   \langle \bE, \nabla\times\bF \rangle = i\omega\mu \langle \bH, \bF \rangle,
   \quad\forall\bF\in\cD(\bbR^3)^3.
\]
Substituting $\bE$ gives
\[
   \langle \epsilon^{-1} \nabla \times \bH , \nabla\times\bF \rangle = 
   \omega^2\mu \langle \bH, \bF \rangle
\]
which clearly implies \eqref{maxwell1}.
{To finish the proof of (i),} it remains to establish the regularity of $\bH$.
Equation \eqref{maxwell02} implies that $\nabla \times \bH \in \bbL^2_\loc(\bbR^3)$.  Together with \eqref{maxwell2}, for any cut-off function $\chi \in \cD(\bbR^3)$, we then have
$$
	\chi \bH \in \bbL^2(\bbR^3), \quad 
	\nabla\cdot(\chi \bH) \in \rL^2(\bbR^3), \quad
	\nabla \times (\chi \bH) \in \bbL^2(\bbR^3).
$$
The Fourier transform then yields $\nabla(\chi \bH) \in
(\rL^2(\bbR^3))^9$ (cf.\ \cite[Ch. 1,
Lem. 2.5]{giraultraviart}). Therefore $\bH \in \bbH^1_\loc(\bbR^3)$.
{The proof of (ii) is then immediate.} 
\end{proof}

Note that we have shown the following result.
\begin{lemma}
\label{lem:1}
There holds the embedding 
\[
   \{ \bv \in \bbL^2_{\loc}(\bbR^3) : 
	 \nabla \times \bv \in \bbL^2_{\loc}(\bbR^3), \; \nabla\cdot \bv \in \rL^2_{\loc}(\bbR^3) \}
	 \subset \bbH^1_\loc(\bbR^3).
\]
\end{lemma}

\section{Periodic electric permittivity}
\label{sec:periodic}
In this section, we assume that $\epsilon$ is periodic with respect to the lattice $\cKt$, cf.\ \eqref{eq:epsgen} and that
\begin{equation}
\label{eq:epsb}
   \epsilon>0,\quad \epsilon\in\rL^{\!\infty}(\bbR^3), 
   \quad \epsilon^{-1}\in\rL^{\!\infty}(\bbR^3).
\end{equation}
The operator $\bH\mapsto \nabla \times ( \tsfrac{1}{n^2} \nabla \times \bH )$ appearing in \eqref{maxwell1} defines an unbounded self-adjoint operator $L$ on $\bbL^2(\bbR^3)$ with domain
\[
   D(L) := \lbrace \bv \in \bbL^2(\bbR^3) : \nabla \times \bv \in \bbL^2(\bbR^3)\ \ 
   \mbox{and}\ \ 
   \nabla \times ( \tsfrac{1}{n^2} \nabla \times\bv) \in \bbL^2(\bbR^3)\rbrace .
\]
We find the relevant part of the spectrum, by imposing in addition the gauge condition \eqref{maxwell2}.
The form domain of the operator (still denoted by $L$) is then
$$
	\bV := \{ \bv \in \bbL^2(\bbR^3) : \nabla \times \bv \in \bbL^2(\bbR^3), \; 
	\nabla \cdot \bv = 0 \}.
$$

Due to the periodicity of $\epsilon$ (equivalently, the periodicity of
$n$) we can apply the Floquet-Bloch transform to $L$ (cf. \cite{ashcroft,joannopoulosbook,kuchmentbook,kuchmentarticle}).
This reduces the problem from studying the operator $L$ to a family of operators $L_{\bxi}$ acting on periodic functions on a compact manifold, the period cell $Q=\bbR^3/\cKt$.  The key result from Floquet-Bloch theory that is used to find the spectrum of $L$ is
\beq
\label{keyfloquetgen}
	\sigma(L) = \overline{\bigcup_{\bxi \in \cB} \sigma(L_{\bxi})},
\eeq
where $\sigma(\cdot)$ denotes the spectrum of an operator.  See \cite{kuchmentbook,kuchmentarticle} and references therein for details.  In particular, \cite{figotin1997} has a good description of Floquet-Bloch theory for the Maxwell operator in $\bbR^3$ where $n$ is periodic on a three-dimensional lattice.  

Let us quickly describe the set $\cB \subset \bbR^3$ that appears in
\eqref{keyfloquetgen}. It is the $1^{\rm st}$ Brillouin zone of $\cK$ (i.e. the Wigner-Seitz period cell of the reciprocal lattice, for definitions see e.g. \cite{ashcroft}).  If $\ba_1 = (\ell_1,0,0)$, $\ba_2 = (0,\ell_2,0)$, and $\ba_3=(0,0,\ell_3)$, then 
$$
   Q = (-\tsfrac{\ell_1}{2},\tsfrac{\ell_1}{2}] \times
   (-\tsfrac{\ell_2}{2},\tsfrac{\ell_2}{2}] \times(-\tsfrac{\ell_3}{2},\tsfrac{\ell_3}{2}] 
   \ \ \mbox{and}\ \ 
   \cB = (-\tsfrac{\pi}{\ell_1},\tsfrac{\pi}{\ell_1}] \times
   (-\tsfrac{\pi}{\ell_2},\tsfrac{\pi}{\ell_3}] \times
   (-\tsfrac{\pi}{\ell_3},\tsfrac{\pi}{\ell_3}] .
$$ 
The details of the Floquet-Bloch transform may be hidden by simply requiring that we search for magnetic fields of the form
\begin{equation}
\label{eq:Hu}
	\bH(\bx) = \bu(\bx) \rme^{i \bxi \cdot \bx}
\end{equation}
where $\bx = (x,y,z)$, $\bxi=(\xi_1,\xi_2,\xi_3) \in \cB$ and $\bu$ is periodic with respect to $\cK$.
In this way $L$ is associated with the family of operators, 
\beq
\label{zinv1}
	\bigl\{ L_{\bxi} : \bxi \in \cB \bigr\},
\eeq
parameterised by $\bxi$, where each $L_{\bxi}$ 
is the self-adjoint operator 
$$
	L_{\bxi} := \nabla_{\bxi} \times (\tsfrac{1}{n^2} \nabla_{\bxi} \times \cdot) 
	\quad\mbox{with}\quad
	\nabla_{\bxi} := 
	(\tsfrac{\partial}{\partial x}, \tsfrac{\partial}{\partial y}, \tsfrac{\partial}{\partial z}) 
	+ i\bxi
$$
operating on a Hilbert space of periodic functions (form domain):
\begin{equation}
\label{eq:Vxi}
	\bV_{\bxi} := \{ \bv \in \bbL^2_{\per} : 
	\nabla_{\bxi} \times \bv \in \bbL^2_{\per}, \; \nabla_{\bxi} \cdot \bv = 0 \}
\end{equation}
where $\bbL^2_{\per} := \{ \bv \in \bbL^2_{\loc}(\bbR^3) : \bv \mbox{
  is periodic on $\cK$} \}$. Note that $\bbL^2_{\per}(\bbR^3)$
identifies with $\bbL^2(Q)$. Likewise we define $\bbH^1_{\per}$ as $\{
\bv \in \bbH^1_{\loc}(\bbR^3) : \bv \mbox{ is periodic on $\cK$} \}$,
which identifies with $\bbH^1(Q)$. Let $a_{\bxi}: \bV_{\bxi} \times
\bV_{\bxi} \rightarrow \bbC$ be the sesquilinear form associated with
$L_{\bxi}$, i.e.
\begin{equation*}
	a_{\bxi}(\bu,\bv) := \int_Q \tsfrac{1}{n^2} \nabla_{\bxi}
        \times \bu \cdot \overline{ \nabla_{\bxi} \times \bv} \; \dd x \dd y \dd z,
        \quad \bu,\bv\in\bV_{\bxi} \,.
\end{equation*}
Here follows the main result of this section.

\begin{theorem}
\label{th:2}
Let the magnetic permeability $\mu=\mu_0$ be constant and let the electric permittivity $\epsilon$ be periodic \eqref{eq:epsgen} and satisfy the boundedness conditions \eqref{eq:epsb}. Then for any $\bxi$ in the Brillouin zone $\cB$, the variational space $\bV_{\bxi}$ \eqref{eq:Vxi} is contained in the Sobolev space $\bbH^1_{\per}$. Moreover $a_{\bxi}$ satisfies a G\r{a}rding inequality on $\bV_{\bxi}$ with respect to the $\bbH^1_{\per}$-norm: for any $\bxi\in\cB$ and $\bv \in \bV_{\bxi}$
\begin{equation}
\label{eq:coer}
   \nrm{n}^2_{\rL^\infty_{\per}} \, a_{\bxi}(\bv,\bv) + (3|\bxi|^2+1)\nrm{\bv}_{\bbL^2_{\per}}^2 \geq 
   \tsfrac12\nrm{\bv}_{\bbH^1_{\per}}^2\,.
\end{equation}
The spectrum of $L_{\bxi}$ is discrete and formed by a sequence of nonnegative eigenvalues.
\end{theorem}

\begin{proof}
We first notice that $\bV_{\bxi}$ is embedded in the space
\[
   \{ \bv \in \bbL^2_{\per} : 
	 \nabla_{\bxi} \times \bv \in \bbL^2_{\per}, \; \nabla_{\bxi}\cdot \bv \in \rL^2_{\per} \}
= 
   \{ \bv \in \bbL^2_{\per} : 
	 \nabla \times \bv \in \bbL^2_{\per}, \; \nabla\cdot \bv \in \rL^2_{\per} \},
\]
which is itself contained in
$
   \{ \bv \in \bbL^2_{\loc}(\bbR^3) : 
	 \nabla \times \bv \in \bbL^2_{\loc}(\bbR^3), \; \nabla\cdot \bv \in \rL^2_{\loc}(\bbR^3) \}.
$
According to Lemma \ref{lem:1}, the latter space is embedded in $\bbH^1_\loc(\bbR^3)$, which proves that $\bV_{\bxi}\subset\bbH^1_{\per}$. For the G\r{a}rding inequality we first write
\[
   \nrm{n}^2_{\rL^\infty_{\per}} \, a_{\bxi}(\bv,\bv) \ge \nrm{\nabla_{\bxi}\times\bv}_{\bbL^2_{\per}}^2 \quad \mbox{for all $\bv \in \bV_{\bxi}$.}
\]
For any $\bv \in \bbH^1_{\per}$, we may integrate by parts over $Q$ to obtain
\[
   \nrm{\nabla_{\bxi}\times\bv}_{\bbL^2_{\per}}^2 + 
   \nrm{\nabla_{\bxi}\cdot\bv}_{\bbL^2_{\per}}^2 = 
   - \langle \Delta_{\bxi}\bv,\bv\rangle\,,
\]
where $\langle \cdot,\cdot\rangle$ is the duality pairing between $\bbH^1_{\per}$ and its dual, and 
$$
   \Delta_{\bxi} := (\tsfrac{\partial}{\partial x}+i\xi_1)^2 +
   (\tsfrac{\partial}{\partial y}+i\xi_2)^2 +
   (\tsfrac{\partial}{\partial z}+i\xi_3)^2.
$$
Another integration by parts yields for any $\bv\in\bbH^1_{\per}$
\[
   \nrm{\nabla_{\bxi}\bv}_{\bbL^2_{\per}}^2  = 
   - \langle \Delta_{\bxi}\bv,\bv\rangle\,.
\]
Hence, for any $\bv\in\bV_{\bxi} \subset \bbH^1_{\per}$ there holds $\nrm{\nabla_{\bxi}\times\bv}_{\bbL^2_{\per}}^2 = \nrm{\nabla_{\bxi}\bv}_{\bbL^2_{\per}}^2$, which implies
\[
   \nrm{n}^2_{\rL^\infty_{\per}} \, a_{\bxi}(\bv,\bv) \ge \nrm{\nabla_{\bxi}\bv}_{\bbL^2_{\per}}^2.
\]
By the arithmetic-geometric mean inequality we find that
\[
   \nrm{\nabla_{\bxi}\bv}_{\bbL^2_{\per}}^2 \ge 
   \tsfrac12\nrm{\nabla\bv}_{\bbL^2_{\per}}^2 - 3|\bxi|^2 \nrm{\bv}_{\bbL^2_{\per}}^2,
\]
which proves \eqref{eq:coer}. The spectral properties of $L_{\bxi}$ are now a classical consequence of the compact embedding of $\bbH^1(Q)$ into $\bbL^2(Q)$. 
\end{proof}

\begin{remark}
We have shown the identity
\[
   \nrm{\nabla_{\bxi}\times\bv}_{\bbL^2_{\per}}^2 + 
   \nrm{\nabla_{\bxi}\cdot\bv}_{\bbL^2_{\per}}^2 = 
   \nrm{\nabla_{\bxi}\bv}_{\bbL^2_{\per}}^2, \quad \text{for any} \ \  \bv\in\bbH^1_{\per}\,.
\]
This can be compared with \cite[Theorem 2.3]{CostabelDauge99}. Note that in our case, there is no boundary, which greatly simplifies the analysis.
\end{remark}

\section{Translationally invariant periodic electric permittivity}
\label{sec:pertrans}
In this section, in addition to the boundedness condition \eqref{eq:epsb}, we assume that $\epsilon$ is translationally invariant in the $z$ direction and periodic with respect to the two-dimensional lattice $\cKp$, see \eqref{eq:epstrans}.

The invariance with respect to $z$ can be seen as periodicity of period $0$ in that direction, leading to an unbounded Brillouin zone $\cB=\cBp\times\bbR$ where $\cBp\subset\bbR^2$ is the first Brillouin zone of the two-dimensional lattice $\cKp$.  Likewise the Floquet-Bloch transform degenerates into the usual Floquet-Bloch transform in variables $\bx'=(x,y)$ and partial Fourier transform in variable $z$, \cite[Annexe B]{flissthese}.  This is the reason why the reduction of problem \eqref{maxwell12} to a family of operators $L_{\bxi}$ with compact resolvent has now two steps:
\begin{enumerate}
\item Consider the Ansatz for any chosen constant $\beta\in\bbR$
\beq
\label{zinv0}
	\bH(x,y,z) = \bh(x,y) \,\rme^{i \beta z},
\eeq
to obtain a problem that is posed on $\bbR^2$ instead of $\bbR^3$.
Some authors (e.g. \cite{dobson99,giani12,soussi}) assume further that
$\beta = 0$ so that Maxwell's equations decouple into the so-called
\emph{transverse electric} (TE) and \emph{transverse magnetic} (TM)
mode problems.  We do not make this assumption here. Neither will we
reduce the problem to the linear Schr\"odinger equation as in \cite{norton2010}.
\item The constant $\beta$ being chosen, perform the Floquet-Bloch transform in transverse variables $\bx'$, leading to the family of operators, 
\beq
\label{zinv1p}
	\bigl\{ L\xib : \bxip \in \cBp \bigr\},
\eeq
parameterised by $\bxip$, where $L\xib$ 
is the self-adjoint operator 
$$
	L\xib := \nabla\xib \times \left(\tsfrac{1}{n^2} \nabla\xib \times\, \cdot\,\right) 
	\quad \mbox{and} \quad 
	\nabla\xib := 
	(\tsfrac{\partial}{\partial x}, \tsfrac{\partial}{\partial y}, 0) 
	+ i(\bxi'\!,\beta),
$$
operating on the space of periodic functions in two dimensions:
\begin{equation}
\label{eq:Vxip}
	\bV\xib := \{ \bv \in \bbL^2(Q') : 
	\nabla\xib \times \bv \in \bbL^2(Q'), \; \nabla\xib \cdot \bv = 0 \}
\end{equation}
Here $Q'=\bbR^2/\cKp$ and $\bbL^2(Q')$ identifies with the two-dimensional space $\bbL^2_{\per}$ relative to the lattice $\cKp$.
\end{enumerate}

\noindent
Let $a\xib : \bV\xib \times \bV\xib \rightarrow \bbC$ be the sesquilinear form associated with $L\xib$
\begin{equation*}
	a\xib(\bu,\bv) := \int_{Q'} \tsfrac{1}{n^2}\, \nabla\xib
        \times \bu \ \cdot\ \overline{ \nabla\xib \times \bv} \; \dd x \dd y,
        \quad \bu,\bv\in\bV\xib \,.
\end{equation*}
The main result of this section is very similar to Theorem \ref{th:2}.

\begin{theorem}
\label{th:3}
Let the magnetic permeability $\mu=\mu_0$ be constant and let the electric permittivity $\epsilon$ be periodic, translationally invariant \eqref{eq:epstrans}, and satisfy the boundedness conditions \eqref{eq:epsb}. Then for any $\beta\in\bbR$ and any $\bxip$ in the Brillouin zone $\cBp$, the variational space $\bV\xib$ \eqref{eq:Vxip} is contained in the Sobolev space $\bbH^1(Q')$. Moreover $a\xib$ satisfies a G\r{a}rding inequality: for any $\bv \in \bV\xib$\,, $\bxip \in \cBp$ and $\beta \in \bbR$
\begin{equation*}
   \nrm{n}^2_{\rL^\infty(Q')} \, a\xib(\bv,\bv) + (3|\bxip|^2+1)\nrm{\bv}_{\bbL^2(Q')}^2 \geq 
   \tsfrac12\nrm{\bv}_{\bbH^1(Q')}^2 + \beta^2 \nrm{\bv}_{\bbL^2(Q')}^2.
\end{equation*}
The spectrum of $L\xib$ is discrete and formed by a sequence of nonnegative eigenvalues.
\end{theorem}

The proof follows the same lines as Theorem \ref{th:2}. Since the embedding of $\bbH^1(Q')$ into $\bbL^2(Q')$ is compact, we deduce the spectral properties of $L\xib$.

\section{Piecewise constant and periodic permittivity on a polyhedral partition}
\label{sec:reg}

Here we return to the primitive Maxwell equations \eqref{maxwell012} and drop for the moment the assumption of translational invariance again. We assume that $\epsilon$ and $\mu$ are piecewise constant and periodic with respect to the lattice $\cKt$, which determines a periodic partition $\cP= \{ P_j \}_{j=1}^J$ of $\bbR^3$ into a finite set of unbounded Lipschitz {\em polyhedral domains}\footnote{We call [Lipschitz] polyhedral domain any [Lipschitz] open set with piecewise plane boundary. The singular points of the boundary form the edges and the corners. } $P_1,\dotsc,P_J$, such that
\begin{subequations}
\label{period}
\begin{equation}
\label{period1}
	\bbR^3 = \overline{ \cup_j P_j },\quad 
	P_j+\bk=P_j,\  \forall\bk\in\cKt,\ \forall j,
	\quad \mbox{and} \quad
	P_i \cap P_j = \emptyset \mbox{ if $i \neq j$},
\end{equation}
and 
\begin{equation}
\label{period2}
	\epsilon = \epsilon_j, \quad \mu = \mu_j \ \  \mbox{on $P_j$},\quad
	\mbox{with $\epsilon_j$ and $\mu_j$ positive constants.}
\end{equation}
\end{subequations}
Note that the quotient sets
\[
   Q_j :=P_j/\cKt,\quad j=1,\ldots,J,
\]
make sense and determine a finite partition into (bounded) polyhedral subdomains of the torus $Q=\bbR^3/\cKt$.

We study the regularity of solutions $(\bE,\bH)$ of system \eqref{maxwell012}.  Associated with partition $\cP$ we define the piecewise Sobolev spaces
\begin{gather*}
	\rP\rH^s_\loc(\bbR^3,\cP) := 
	\{ u \in \rL^2_\loc(\bbR^3) : \ \ u|_{P_j} \in \rH^s_\loc(P_j) \}, \qquad s \geq 0,\\
	\rP\rH^s(Q,\cP) := \{ u \in \rL^2(Q) : \ \ u|_{Q_j} \in \rH^s(Q_j) \}, \qquad s \geq 0.
\end{gather*}
The regularity results of \cite{costabel} adapt in the following
way. In view of its use for $\alpha = \epsilon$ or $\alpha = \mu$ let
us make the following definition.

\begin{definition}
\label{def:sigma}
For $\alpha$ piecewise constant and positive on the partition $\cP$ of the torus $Q$,
define an operator $\Delta_\alpha : H^1(Q) \rightarrow H^{-1}(Q)$ by
$$
	\Delta_\alpha u = \nabla \cdot (\alpha \nabla u ),\quad
        \mbox{for all } \ u \in \rH^1(Q).	
$$
Associated with this operator let $\sigma_\alpha$ (for $\alpha = \epsilon$ or $\mu$) be the supremum of $s >0$ ($s \neq 1/2$) such that 
\beq
\label{regeq2}
	u \in \rH^1(Q) \mbox{ and } \Delta_\alpha u \in \rH^{-1+s} 
	\quad \Rightarrow \quad u \in \rP\rH^{1 + s}(Q,\cP).
\eeq
\end{definition}
The values of $\sigma_\alpha$ depend on the \emph{singular exponents} at interface edges and corners.

\begin{remark}
\label{rem5.2}
If $\alpha$ takes two distinct values in the neighbourhood of an
interface edge between two materials, an explicit formula for singular
exponents shows that $\sigma_\alpha<1$ {(cf.~\cite{CostabelStephan85} and \cite[Th.8.1]{costabel})}.
\end{remark}
 
Results on {\em interior regularity} of polyhedral transmission
problems apply to our situation since we do not have any external
boundaries, and thus no exterior boundary conditions.  The following
theorem is adapted from \cite[Thm. 7.1]{costabel}.  As usual, we
denote $\rP\rH^s_\loc(\bbR^3,\cP)^3$ by $\rP\bbH^s_\loc(\bbR^3,\cP)$.

\begin{theorem}
\label{th:4}
Let $(\bE, \bH) \in \bbL^2_{\loc}(\bbR^3) \times \bbL^2_{\loc}(\bbR^3)$ satisfy equations \eqref{maxwell012} with $\omega \neq 0$. Then $\bE$ and $\bH$ have the following regularity:
\begin{align}
\label{eq:E1}
	&\bE \in \rP\bbH^s_{\loc}(\bbR^3,\cP), \qquad 
	\mbox{for all $s < \min \{ \sigma_\epsilon, \sigma_\mu + 1\}$,} \\
\label{eq:H1}
	&\bH \in \rP\bbH^s_{\loc}(\bbR^3,\cP), \qquad 
	\mbox{for all $s < \min \{ \sigma_\mu, \sigma_\epsilon + 1 \}$}.
\end{align}
\end{theorem}

This result is a direct application of \cite[Thm. 7.1]{costabel}. It relies on the analysis of edge and corner singularities.  These depend on the interface edge and corner singularities of the scalar operators $\Delta_\epsilon$ and $\Delta_\mu$.

\begin{remark}\mbox{} {It follows from Remark \ref{rem5.2} that}\\[0.5ex]
1) if $\epsilon$ has interfacial edges then $\sigma_\epsilon < 1$ and so, in general, $\bE \notin \rP\bbH^1_{\loc}(\bbR^3,\cP)$.\vspace{0.5ex}

\noindent
2) Likewise if $\mu$ has interfacial edges then $\sigma_\mu < 1$ and, in general, $\bH \notin \rP\bbH^1_{\loc}(\bbR^3,\cP)$.\vspace{0.5ex}

\noindent
3) If $\omega=0$, the regularity results \eqref{eq:E1}-\eqref{eq:H1}
still hold provided that we complete the system \eqref{maxwell012} by
the gauge conditions $\Div\epsilon\bE=0$ and $\Div\mu\bH=0$ {(cf.~\cite{costabel})}.
\end{remark}

In the case of a non-magnetic material, $\mu=\mu_0$. Since $\mu$ is constant the operator $\Delta_\mu$ (Definition \ref{def:sigma}) is simply $\mu_0 \Delta : H^1(Q) \rightarrow H^{-1}(Q)$. Thus we are reduced to the ordinary Laplace operator. The standard theory of elliptic operators yields that there does not exist an upper bound on the $s$ for which \eqref{regeq2} holds, so we may formally take $\sigma_\mu = \infty$ and the regularity result becomes the following.

\begin{corollary}
\label{cor:1th:3}
Let $(\bE, \bH) \in \bbL^2_{\loc}(\bbR^3) \times \bbL^2_{\loc}(\bbR^3)$ satisfy equations \eqref{maxwell012} with $\omega \neq 0$. We assume that $\mu=\mu_0$. Then $\bE$ and $\bH$ have the following regularity:
\begin{align}
\label{eq:E2}
	&\bE \in \rP\bbH^s_{\loc}(\bbR^3,\cP) \qquad 
	\mbox{for all $s < \sigma_\epsilon$} \\
\label{eq:H2}
	&\bH \in \rP\bbH^s_{\loc}(\bbR^3,\cP) \qquad 
	\mbox{for all $s < \sigma_\epsilon + 1 $}.
\end{align}
\end{corollary}

This result implies that the regularity of the magnetic field may be a whole degree better than the electric field. This is a very good justification for posing the original problem \eqref{maxwell012} in terms of only the magnetic field \eqref{maxwell12}, because we would expect numerical methods to converge faster to the more regular magnetic field.

As a corollary of the previous statement, we obtain the regularity of the eigenvectors of the Floquet operators $L_{\bxi}$,  cf.\ Theorem \ref{th:2}.

\begin{theorem}
\label{th:5}
Let the magnetic permeability $\mu = \mu_0$ be constant and let the electric permittivity be piecewise constant periodic over a polyhedral partition $\cP$ of the periodicity cell $Q$ \eqref{period}. Let $\bxi$ belong to the first Brillouin zone $\cB$. Then any eigenvector $\bU\in\bbH^1(Q)$ of the operator $L_{\bxi}$ (cf.\ Theorem {\em\ref{th:2}}) satisfies
\[
   \bU\in\rP\bbH^{s}(Q,\cP) \quad \mbox{for any \ $s < \sigma_\epsilon + 1$}
\]
with $\sigma_\epsilon$ the number introduced in Definition {\em\ref{def:sigma}}.
\end{theorem}

\begin{proof}
The eigenvector $\bU$ is periodic and satisfies the equation $L_{\bxi}\bU=\kappa^2\bU$ for some real $\kappa$, where we recall
$$
	L_{\bxi} := \nabla_{\bxi} \times (\tsfrac{1}{n^2} \nabla_{\bxi} \times \cdot) 
	\quad\mbox{with}\quad
	\nabla_{\bxi} := 
	(\tsfrac{\partial}{\partial x}, \tsfrac{\partial}{\partial y}, \tsfrac{\partial}{\partial z}) 
	+ i\bxi
$$
Setting (cf. \eqref{eq:Hu})
\[
	\bH(\bx) = \bU(\bx)\, \rme^{i \bxi \cdot \bx}
\]
we obtain that $\bH$ is a Blochwave belonging to the space
$$
	\bbH^1_{\bxi}(\bbR^3) := 
	\{ \bF \in \bbH^1_\loc(\bbR^3) : \ 
	\bF(\bx + \bk) = \rme^{i \bxi \cdot \bk} \bF(\bx)\ \ 
	\forall \bx \in \bbR^3, \ \forall \bk \in \cKt \} 	
$$
and satisfying the equations \eqref{maxwell12}. Setting 
\[
   \bE = \tsfrac{i}{\omega\epsilon} \nabla\times\bH
\]
we obtain a solution of equations \eqref{maxwell012} and may apply Corollary \ref{cor:1th:3}. 
\end{proof}

\section{Translationally invariant, piecewise constant and periodic electric permittivity on a polygonal transverse partition}
\label{sec:reg+}

This section is devoted to deriving a bespoke regularity result for
PCFs.  Here the results in \cite{costabel} can be improved since we
have the special situation where the magnetic permeability $\mu =
\mu_0$ is constant and the electric permittivity $\epsilon = n^2
\epsilon_0$ is invariant with respect to $z$, as well as bi-periodic 
in the $xy$-plane. 

Our aim now is to further refine Theorem \ref{th:4} under these additional assumptions. Our new result will focus on the $x$ and $y$ components of the magnetic field since these will actually have more regularity and are often sufficient for photonic crystal fibre modelling, see e.g. \cite{pearce2005adaptivecurvilinear,norton2012apnum}. The other components of the magnetic and electric field can then be recovered in an easy and more accurate way in a post-processing procedure (see \cite{nortonthesis} for details). It also complements the regularity theory for the $z$-components of $\bE$ and $\bH$, developed under a supplementary assumption in \cite{elschner}. We will come back to this in the next section.

In addition to assumptions \eqref{period} ($\epsilon$ periodic and piecewise constant on a polyhedral partition), let us also assume that $\epsilon$ is translationally invariant with respect to $z$ so that
$$
	\epsilon(\bx) = \epsilon(\bx'), \qquad 
	\mbox{for all \ $\bx = (x,y,z) \in \bbR^3$, \ with \ $\bx' = (x,y)$}.
$$
Then, the sets $P_j$ of the polyhedral partition are translationally invariant too, so they have the form
\[
   P_j = P'_j\times\bbR,\quad \ \mbox{with \ $P'_j$ \ polygonal in \ $\bbR^2$.}
\]
We denote by $\cP'$ the partition $\{ P'_j \}_{j=1}^J$ and by $\rP\rH^s(\cdot,\cP')$ the corresponding piecewise Sobolev spaces:
\begin{gather*}
	\rP\rH^s_\loc(\bbR^2,\cP') := 
	\{ u \in \rL^2_\loc(\bbR^2) : \ \ u|_{P'_j} \in \rH^s_\loc(P'_j) \}, \qquad s \geq 0,\\
	\rP\rH^s(Q',\cP') := \{ u \in \rL^2(Q') : \ \ u|_{Q'_j} \in \rH^s(Q'_j) \}, \qquad s \geq 0.
\end{gather*}
Here $Q'$ is the two-dimensional torus $\bbR^2/\cKp$ and $Q'_j=P'_j/\cKp$.

As we have seen in \S\ref{sec:pertrans}, in the translationally invariant case we may reduce problem \eqref{maxwell12} to a simpler problem by using the Ansatz \eqref{zinv0} with fixed $\beta \in \bbR$,
$$
	\bH(\bx) = \bh(\bx') \,\rme^{i \beta z}.
$$
Then, avoiding again to enter into the details of the Floquet-Bloch transform, 
we may further simplify the problem by setting 
\[
	\bh(\bx') = \bu(\bx')\, \rme^{i \bxip \cdot \bx'}, \quad \text{for some} \ \  \bxip \in \mathcal{B}',
\]
where $\bu$ is periodic and an eigenvector of the Floquet operator
$$
	L{\xib} := \nabla\xib \times \left(\tsfrac{1}{n^2} \nabla\xib \times\, \cdot\,\right) 
	\quad\mbox{with}\quad  
	\nabla\xib := 
	(\tsfrac{\partial}{\partial x}, \tsfrac{\partial}{\partial y}, 0) 
	+ i(\bxip\!,\beta) 
$$
and satisfies thus the equation $L\xib\bu=\kappa^2\bu$, for some real $\kappa$
(cf. Theorem~\ref{th:2}). Recalling from Theorem \ref{th:3} that $\bu \in \bbH^1(Q')$ we obtain that $\bh$ belongs to the space
$$
	\bbH^1_{\bxip}(\bbR^2) := 
	\{ \boldf \in \bbH^1_\loc(\bbR^2) : \ 
	\boldf(\bx' + \bk') = \rme^{i \bxip \cdot \bk'} \boldf(\bx')\ \ 
	\forall \bx' \in \bbR^2, \ \forall \bk' \in \cKp \} 	
$$
and satisfies the following equations in $\bbR^2$:
\begin{subequations}
\label{reg3}
\begin{align}
	\nabla'_{(0,0,\beta)} \times (\tsfrac{1}{n^2} \nabla'_{(0,0,\beta)} \times \bh ) 
	&= \kappa^2 \bh, \label{reg3a} \\
	\nabla'_{(0,0,\beta)} \cdot \bh &= 0. \label{reg3b}
\end{align}
\end{subequations}
Note that $\epsilon = \epsilon_0 n^2$ so that $n^2$ is piecewise constant on the same partition $\cP'$ as $\epsilon$.

Now recall that $\nabla'_{(0,0,\beta)} = (\tsfrac{\partial}{\partial x},\tsfrac{\partial}{\partial y},0) + i (0,0,\beta)$ and expand \eqref{reg3} to get
\begin{subequations}
\label{reg4}
\begin{align}
	\partial_y \tsfrac{1}{n^2} ( \partial_x h_y - \partial_y h_x ) - i \beta \tsfrac{1}{n^2} (i \beta h_x - \partial_x h_z) &= \kappa^2 h_x, \label{reg4a}\\
	-\partial_x \tsfrac{1}{n^2} ( \partial_x h_y - \partial_y h_x) + i \beta \tsfrac{1}{n^2} (\partial_y h_z - i \beta h_y) &= \kappa^2 h_y, \label{reg4b}\\
	\partial_x \tsfrac{1}{n^2} ( i \beta h_x - \partial_x h_z) - \partial_y \tsfrac{1}{n^2} (\partial_y h_z - i \beta h_y) & = \kappa^2 h_z, \label{reg4c}\\
	\partial_x h_x + \partial_y h_y + i \beta h_z = 0. \label{reg4d}
\end{align}
\end{subequations}
Using the $\rH^1$ regularity of $\bh$, it follows from \eqref{reg4a}, \eqref{reg4b} and \eqref{reg4d} that
\begin{align*}
	\partial_y \tsfrac{1}{n^2} ( \partial_x h_y - \partial_y h_x) & \in L^2_\loc(\bbR^2),\\
	\partial_x \tsfrac{1}{n^2} ( \partial_x h_y - \partial_y h_x) & \in L^2_\loc(\bbR^2),\\
	\partial_x h_x + \partial_y h_y & \in H^1_\loc(\bbR^2).
\end{align*}
Let us define $\boldf\,' = (f_x,f_y) \in \rL^2_\loc(\bbR^2)^2$ by the formulas
\begin{subequations}
\label{eq:bf}
   \begin{align}
   f_x = \kappa^2 h_x + 
   i \beta \tsfrac{1}{n^2} (i \beta h_x - \partial_x h_z) + i\beta\partial_x h_z, \\
   f_y = \kappa^2 h_y - 
   i \beta \tsfrac{1}{n^2} (\partial_y h_z - i \beta h_y) +  i\beta\partial_y h_z. 
   \end{align}
\end{subequations}
Hence, $\bh' = (h_x,h_y)$ satisfies the elliptic system,
\begin{align*}
	\partial_y \tsfrac{1}{n^2} ( \partial_x h_y - \partial_y h_x) 
	- \partial_x ( \partial_x h_x + \partial_y h_y) & = f_x,\\
	- \partial_x \tsfrac{1}{n^2} ( \partial_x h_y - \partial_y h_x) 
  -  \partial_y ( \partial_x h_x + \partial_y h_y) & = f_y,
\end{align*}
or in terms of two-dimensional $\Curl$, $\rm{curl}$, $\grad$ and $\Div$ operators
$$
	\Curl \tsfrac{1}{n^2} \curl \bh' - \grad \Div \bh' = \boldf\,'.
$$ 

We have thus reduced the regularity analysis for the eigenproblem \eqref{maxwell12} to the regularity analysis of the following problem
\begin{equation}
\label{eq:bh'}
   \bh' \in \rH^1_\loc(\bbR^2)^2 \quad\mbox{such that}\quad
   M \bh' = \boldf\,'\ \  \mbox{for some} \ \ \boldf\,' \in \rL^2_\loc(\bbR^2)^2
\end{equation}
where $M$ is the $2\times2$ operator
\begin{equation}
\label{eq:M}
   M := \Curl \tsfrac{1}{n^2} \curl - \grad \Div,
\end{equation}
acting on all of $\bbR^2$ and the refractive index $n$ is translationally invariant (in $z$) and piecewise constant periodic over a polygonal partition $\cP'$ of the periodicity cell $Q'$.

The operator $M$ defines an elliptic system and the bottleneck for optimal regularity comes from the corners of the subdomains $Q'_j$. Outside any neighbourhood of the corners we have optimal piecewise regularity, i.e. $\rP\rH^2$, but the norm may blow up near the corners. We now investigate the strength of the corner singularities. 
Our approach is to analyse the operator $M$ using the Kondrat'ev method
\cite{kondratiev} near the corners. We need some notation. Let $\cC$ be the set of corners $\bc$ of all subdomains $Q'_j$.

Let us choose a corner $\bc$. It suffices to prove $\rP\rH^{2 - \eta}$ regularity in a neighbourhood of this corner.  After a possible reordering of the subdomains, let us denote by $Q'_\ell$, $\ell=1,\ldots,\cL$, the subdomains containing $\bc$ in their boundaries.
In polar coordinates $(r,\theta)$ centred at $\bc$, we may assume that, for $r_0$ small enough
\begin{subequations}
\label{omegaell}
\begin{equation}
\label{omegaell1}
   Q'_\ell \cap \cB(\bc,r_0) = \{\bx\in\bbR^2 :\ \theta\in\Omega_\ell
   = (\omega_{\ell-1},\omega_{\ell}), \ r\in(0,r_0)\}
\end{equation}
with
\begin{equation}
\label{omegaell2}
   0=\omega_0<\omega_1<\ldots<\omega_{\cL}=2\pi.
\end{equation}
\end{subequations}
The Kondrat'ev method is very general and also applies to transmission problems, see \cite{Nicaise93,NicaiseSandig94}.  To obtain our desired regularity, we must prove that the Mellin symbol of $M$ at the corner $\bc$ is invertible in a certain strip of the complex plane. This method was further developed in \cite{dauge88} where equivalent, or more adapted, conditions are exhibited. This method was applied to Maxwell equations in \cite{costabel00,costabel}. Let us explain the latter method in the case of the operator $M$.

For $\lambda \in \bbC$, we need a space of {\em quasi-homogeneous functions of degree $\lambda$
and angular regularity $m\in\bbN\cup\{0\}$}
\begin{equation}
	S^\lambda_{[m]} := \Big\{ \Phi(r,\theta) = \sum_{q=0}^\mathcal{Q} r^\lambda \log^q r\, \phi_q(\theta) : 
	\\[-2.ex]
	\mathcal{Q} \in \bbN, \ \  \phi_q\in\rH^m(\bbR/2\pi\bbZ)\Big\}.
\end{equation}
The integer $\mathcal{Q}$ plays the role of a polynomial degree. It
{is problem specific} and cannot be specified in the general Ansatz.
Denote by $\bm{S}^\lambda_{[m]}$ the product $S^\lambda_{[m]} \times S^\lambda_{[m]}$.

We note that for $m=1$, functions $\bm{\Phi}\in \bm{S}^\lambda_{[1]}$ belong to $\rH^1(\cB(\bc,r_0))^2$ for $\Re\lambda>0$, but, even when $m\ge2$, do not belong to $\rP\rH^2(\cB(\bc,r_0),\cP')^2$ if $\Re\lambda<1$.
The singularities of problem \eqref{eq:bh'} are now sought in the space $\bm{S}^\lambda_{[1]}$, with $0<\Re\lambda<1$.  Define the space $\bm{Z}_M^\lambda$ to be
\begin{equation}
\label{eq:Z}
	\bm{Z}_M^\lambda := \{ \bm{\Phi} \in \bm{S}^\lambda_{[1]} : \ \ M \bm{\Phi} = 0 \}
	\quad\mbox{with}\quad 0<\Re\lambda<1.
\end{equation}
The ellipticity of $M$ implies that $\bm{Z}_M^\lambda$ is reduced to
$\{0\}$ for all but a finite set of $\lambda$ and that for each of these values of $\lambda$, $\bm{Z}_M^\lambda$ is finite dimensional. Then
\cite[\S1]{kondratiev} adapted to Sobolev spaces with real exponents
as in \cite{dauge88} implies that the solution $\bh'$ of \eqref{eq:bh'} expands\footnote{The introduction of a non-optimal regularity for the regular part ``$\bh'_0$ in $\rP\rH^{2-\eta}$ for all $\eta>0$'' allows a simplification of the statement. A sharp regularity ``$\bh'_0$ in $\rP\rH^{2}$'' would require a condition of invertibility of the Mellin symbol on the line $\Re\lambda=1$, or, more precisely, the condition of {\em injectivity modulo polynomials} as introduced in \cite{dauge88}. As we will see in the next section, the latter condition is not satisfied, whence the interest of the weak regularity statement for the regular part.} around the corner $\bc$ as
\begin{multline}
\label{eq:exp1}
   \bh' = \sum_{0<\Re\lambda<1} \bm{\Phi}^\lambda + \bh'_0 \\[-1ex] 
   \mbox{with}\quad
   \bm{\Phi}^\lambda\in\bm{Z}_M^\lambda \ \ \mbox{and}\ \ 
   \bh'_0\in\rP\rH^{2-\eta}(\cB(\bc,r_0),\cP')^2\ \ \forall\eta>0.
\end{multline}

However, we will now show that for the operator $M$ in \eqref{eq:M} the 
set $\bm{Z}_M^\lambda$ is reduced to $\{0\}$, for all $\lambda\in\bbC$ 
in the strip 
$0<\Re\lambda<1$, which immediately implies the following regularity result.

\begin{theorem}
\label{th:7}
Let $\beta\in\bbR$ and let $\bxip$ belong to the first Brillouin zone $\cB'$.
Assume that the refractive index $n$ is 
translationally invariant (in $z$) and piecewise constant periodic 
over a polygonal partition $\cP'$ of the periodicity cell $Q'$.
Now, consider the solution $\bh' \in \rH^1_\loc(\bbR^2)^2$ of \eqref{eq:bh'}
with $M$ as defined 
in \eqref{eq:M}, or equivalently the $x$ and $y$ components $\bh'=(h_x,h_y)$ 
of the solution $\bh\in\bbH^1_\loc(\bbR^2)$ of \eqref{reg3}.
Then
$$
	\bh' \in \rP\rH^{2 - \eta}_{\loc}(\bbR^2,\cP')^2 
	\qquad \mbox{for any \ $\eta>0$.}
$$
\end{theorem}

To prove this theorem we will need to also introduce the spaces of quasi-homoge\-neous functions associated with the scalar Laplace operator: For $\mu\in\bbC$ let
\begin{equation}
\label{eq:SDelta}
   S_\Delta^\mu := \{\Phi \in S^\mu_{[0]} :\ \Delta\Phi = 0 
   \ \ \mbox{in}\ \ \bbR^2\setminus\{0\} \}.
\end{equation}
Again, the ellipticity of $\Delta$ implies that $S_\Delta^\mu$ is finite dimensional, and reduced to $\{0\}$ for all but a finite set of $\mu$. Each $\mu$ such that $S_\Delta^\mu$ is not trivial is associated with a maximal value $\mathcal{Q}^\mu_{\max}$ of the degree $\mathcal{Q}$. Moreover, still by elliptic regularity, $S_\Delta^\mu$ is contained in $S^\mu_{[m]}$ for any $m\in\bbN$.
These spaces are analytically known\footnote{Though Lemma \ref{lem:SDelta} would be difficult to find in this form in the literature, its proof is very classical and relies on the separation of variables in polar coordinates, exactly like for the Dirichlet or Neumann problem in a plane sector \cite[\S5]{kondratiev}, see also \cite[\S2-3]{ddkppp2013}. Note also that the appearance of integers (and polynomial functions) here is due to the fact that the equation $\Delta\Phi=0$  in $\bbR^2\setminus\{0\}$ implies that it is satisfied on the whole plane $\bbR^2$ as soon as $\Re\lambda>0$.}:
\begin{lemma}
\label{lem:SDelta}
\begin{enumerate}
\item If $\mu$ is not an integer, {then} $S_\Delta^\mu=\{0\}$,
\item If $\mu\in\bbN$, {then $\mathcal{Q}^\mu_{\max}=0$ and}
  $S_\Delta^\mu$ is the space of harmonic polynomials {that are} homogeneous of degree $\mu$.
\item If $\mu=0$, {then $\mathcal{Q}^\mu_{\max}=1$ and} $S_\Delta^\mu$ is generated by $1$ and $\log r$.
\item If $\mu<0$ and is an integer, {then $\mathcal{Q}^\mu_{\max}=0$
  and} $r^{-2\mu}\Phi$ is an harmonic polynomial.
\end{enumerate}
\end{lemma}

\begin{proof}[of Theorem \ref{th:7}]
If we can show that $\bm{Z}_M^\lambda = \{0\}$, for all $\lambda\in\bbC$ in the strip $0<\Re\lambda<1$, then the result follows immediately from \eqref{eq:exp1}.

Let $0 < \Re \lambda < 1$ and suppose $\bm{\Phi} \in \bm{Z}_M^\lambda$.  Then $M \bm{\Phi} = 0$.  Define $\Psi = \curl \bm{\Phi}$ and $\Pi = \Div \bm{\Phi}$.  Then $\Psi$ and $\Pi$ are quasi-homogeneous of degree $\lambda - 1$ and satisfy
\beq
\label{reg6}
	\Curl \tsfrac{1}{n^2} \Psi - \grad \Pi = 0.
\eeq
Taking the $\curl$ and $\Div$ of \eqref{reg6}, we find
$$
	\Delta(\tsfrac{1}{n^2} \Psi) = 0 \quad \mbox{and} \quad
	\Delta(\Pi) = 0
$$
and so $\tsfrac{1}{n^2} \Psi$ and $\Pi$ belong both to $S_\Delta^{\lambda-1}$.

Since $0 < \Re \lambda < 1$, we have $S_\Delta^{\lambda-1}=\{0\}$ and so $\tsfrac{1}{n^2} \Psi = \Pi = 0$.  Therefore we have proved that $\curl \bm{\Phi}=0$ and $\Div \bm{\Phi}=0$. These equations are valid on $\bbR^2$. The equation $\curl \bm{\Phi}=0$ implies that $\bm{\Phi}$ is a gradient: $\bm{\Phi}=\grad\varphi$ with $\varphi$ quasi-homogeneous of degree $\lambda+1$. The equation $\Div \bm{\Phi}=0$ implies that $\Delta\varphi=0$. Thus $\varphi$ belongs to $S_\Delta^{\lambda+1}$. Since $\lambda+1$ is not an integer,  $S_\Delta^{\lambda+1}$ is reduced to $\{0\}$, $\varphi=0$, and finally $\bm{\Phi}=0$. Hence $\bm{Z}_M^\lambda = \{0\}$ and the proof is complete. 
\end{proof}

The following result on the regularity of the
eigenvectors of the related Floquet operators $L\xib$ 
is a simple corollary to Theorem \ref{th:7}.

\begin{corollary}
\label{th:6}
Let the magnetic permeability $\mu=\mu_0$ be constant and let the electric permittivity $\epsilon$ be translation invariant (in $z$) and piecewise constant periodic over a polygonal partition $\cP'$ of the periodicity cell $Q'$. Let $\beta\in\bbR$ and let $\bxip$ belong to the first Brillouin zone $\cB'$. Then any eigenvector $\bu=(u_x,u_y,u_z)\in\bbH^1(Q')$ of the operator $L\xib$ (cf.\ Theorem {\em\ref{th:3}}) satisfies
\[
   \bu' = (u_x,u_y) \in\rP\rH^{2-\eta}(Q',\cP')^2
   \qquad \mbox{for any $\eta>0$.}
\]
\end{corollary}

The following corollary of Theorem \ref{th:7} about the global
regularity of $\bh'$ is a consequence of Grisvard \cite{grisvard} and
Petzoldt \cite[Lemma 2.1]{petzoldt}.

\begin{corollary}
\label{cor:66}
Under the same assumptions as in Theorem \ref{th:7} we have
$$
	\bh'=(h_x,h_y) \in \rH^{3/2-\eta}_{\loc}(\bbR^2)^2 \qquad \mbox{for any $\eta>0$}.
$$
\end{corollary}

\section{Case of simple interface edges between two materials}
\label{sec:reg++}
We consider finally a particular case of the framework studied in the previous section.  In addition to the assumptions of the previous section ($\mu$ is constant and $\epsilon = n^2 \epsilon_0$ is translationally invariant in $z$ and piecewise constant periodic on a polygonal partition $\cP'$ in the $xy$-plane), we also assume that, at each corner $\bc$, only two sectorial regions are touching the corner. This means that in \eqref{omegaell} we have $\cL=2$ and
\[
   0=\omega_0<\omega_1=\omega_\bc<\omega_2=2\pi \quad\mbox{with}\quad \omega_\bc\neq\pi,
\]
and distinct permittivities $\epsilon_1$ in $\Omega_1=(\omega_0,\omega_1)$ and $\epsilon_2$ in $\Omega_2=(\omega_1,\omega_2)$.

Since the material is translationally invariant in the $z$ coordinate let us come back for a while to the primitive equations \eqref{maxwell012} with the Ansatz
$$
	\bE(\bx) = \be(\bx')\, \rme^{i \beta z} \quad \mbox{and} \quad
	\bH(\bx) = \bh(\bx') \,\rme^{i \beta z}.
$$
The Maxwell equations become
\begin{subequations}
\label{reg2}
\begin{align}
	\nabla_{(0,0,\beta)} \times \be - i \omega \mu \bh &= 0, \label{reg2a} \\
	\nabla_{(0,0,\beta)} \times \bh + i \omega \epsilon \be &= 0. \label{reg2b}
\end{align}
\end{subequations}
The regularity and the first singularities of the longitudinal components $e_z$ and $h_z$ are studied in \cite{elschner} by application of the Kondrat'ev method. As before, singularity exponents $\lambda$ are searched for in the strip $0<\Re\lambda<1$. It was proved in \cite[Lem 4.2]{elschner} that for each corner $\bc$, there exists a unique singularity exponent $\lambda_\bc$ in this strip, and that $\lambda_\bc$ is the unique solution of the transcendental equation
\begin{equation}
\label{eq:lamc}
	\frac{ \sin ( (\pi-\omega_\bc)\lambda_\bc)}{\sin(\pi \lambda_\bc)} = 
	\pm \frac{\epsilon_1 + \epsilon_2}{\epsilon_1 - \epsilon_2}
	\quad\mbox{with}\quad 0<\Re\lambda_\bc<1\,.
\end{equation}
We notice that this equation is the same for the scalar transmission problem $\Delta_\epsilon=\nabla \cdot(\epsilon\nabla \cdot)$ \cite{CostabelStephan85} and the common optimal regularity exponent for $\Delta_\epsilon$ and for the problem for the couple $(e_z,h_z)$ is
\begin{equation}
	\sigma_\epsilon = \min_{\bc \in \cC} \lambda_\bc.
\end{equation} 

\begin{theorem}[Elschner, Hinder, Penzel \& Schmidt \cite{elschner}]
\label{th:elschner}
Suppose that $(\be,\bh)$ belongs to $\bbL^2_{\loc}(\bbR^2) \times \bbL^2_{\loc}(\bbR^2)$ and satisfies \eqref{reg2}.  Then, 
$$
	e_z \in \rP\rH^{1+\sigma_\epsilon-\eta}_{\loc}(\bbR^2,\cP') \quad \mbox{and} \quad 
	h_z \in \rP\rH^{1+\sigma_\epsilon-\eta}_{\loc}(\bbR^2,\cP')\ \quad \mbox{for any $\eta>0$},
$$
and in the neighbourhood of each corner $\bc \in \cC$ there exist a constant $\gamma_\bc$ and generic scalar functions $\phi(\theta)$ and $\psi(\theta)$ that are smooth on $[0,\omega_c]$ and $[\omega_c,2\pi]$ such that
\begin{align*}
	e_z &=  \gamma_\bc r^{\lambda_c} \phi(\theta) + e_{z,0} \qquad\mbox{with}\qquad
	e_{z,0} \in\rP\rH^{2-\eta}(\cB(\bc,r_0),\cP')\ \ \forall\eta>0, \\
	h_z &=  \gamma_\bc r^{\lambda_c} \psi(\theta)  + h_{z,0} \qquad\mbox{with}\qquad
	h_{z,0} \in\rP\rH^{2-\eta}(\cB(\bc,r_0),\cP')\ \ \forall\eta>0.
\end{align*}
\end{theorem}

\begin{remark}
\label{rem:elschner}
In the neighbourhood of the corner $\bc \in \cC$, the regularity of $e_z$ and $h_z$ is limited by $\lambda_\bc$:
\[
   e_z,\ h_z \in\rP\rH^{1+\lambda_\bc-\eta}(\cB(\bc,r_0),\cP')\ \quad \mbox{for any $\eta>0$}.
\]
\end{remark}

\begin{remark}
According to \cite[Theorem 8.1]{costabel}, $\lambda_\bc>\frac12$. Therefore, $e_z$ and $h_z$ have global regularity (similarly to Corollary \ref{cor:66}) 
\[
   e_z,\ h_z \in\rH^{3/2-\eta}_{\loc}(\bbR^2)\ \quad \mbox{for any $\eta>0$}.
\]
\end{remark}

Now, relying on Theorem \ref{th:elschner}, we are in a position to further
improve our regularity result of Theorem \ref{th:7} 
on the transverse components of the magnetic field $\bh$.

\begin{theorem}
\label{th:hx+}
Suppose that $(\be,\bh) \in  \bbL^2_{\loc}(\bbR^2) \times \bbL^2_{\loc}(\bbR^2)$ satisfies \eqref{reg2}. Let a corner $\bc \in \cC$ be chosen. Then there exist a constant $\gamma_\bc$ and generic two-component functions $\bm\phi_0(\theta)$ and $\bm\phi_1(\theta)$ that are smooth on $[0,\omega_c]$ and $[\omega_c,2\pi]$, such that the transverse components $\bh'=(h_x,h_y)$ of the magnetic field can be expanded as
\[
  \bh' =  \gamma_\bc r \big\{\bm\phi_0(\theta) +\log r\bm\phi_1(\theta)\big\}
  + \bh'_{0} \ \ \mbox{with}\ \ 
	\bh'_{0} \in\rP\rH^{2+\lambda_\bc-\eta}(\cB(\bc,r_0),\cP')^2\ \ \forall\eta>0. 
\]
Here $\lambda_\bc$ is the singularity exponent defined in \eqref{eq:lamc}.
\end{theorem}

\begin{proof}
Recall from Theorem \ref{th:7}  that $\bh' \in
\rP\rH^{2-\eta}_{\loc}(\bbR^2,\cP')^2$ for any $\eta > 0$
and that $M\bh'=\boldf\,'$ with $M$ defined in \eqref{eq:M} and
$\boldf\,'$ given in \eqref{eq:bf}. Let us choose a corner
$\bc\in\cC$. Relying on Theorem \ref{th:elschner} (and Remark
\ref{rem:elschner}), we see that $\boldf\,'$ is more regular than just
$\rL^2$. In fact,
\[
   \boldf\,' \in \rP\rH^{\lambda_\bc-\eta}(\cB(\bc,r_0),\cP')^2 \quad\mbox{for any $\eta>0$}.
\]
Now, we have an expansion for $\bh'$ like \eqref{eq:exp1} for
$\lambda$ in the strip $0<\Re\lambda<1+\lambda_\bc$. Since this strip
contains also the integer $1$, we have to consider a more general definition for the space $\bm{Z}_M^\lambda$, like in \cite{dauge88,costabel,costabel00}:
Let $\bm{P}^\lambda$ be the space of two-component polynomial functions in $\bx'$ that are homogeneous of degree $\lambda$, then
\begin{equation}
\label{eq:Zgen}
   \bm{Z}_M^\lambda :=\{ \bm{\Phi} \in \bm{S}^\lambda_{[1]} / \bm{P}^\lambda: \quad
   M\bm{\Phi}\in \bm{P}^{\lambda-2}\}.
\end{equation}
We note that $\bm{P}^\lambda = \{0\}$ 
if $\lambda$ is not a natural number. Therefore, as soon as $\lambda
\notin \bbN \cup \{0\}$ this definition of $\bm{Z}_M^\lambda$ reduces
to the original one in \eqref{eq:Z}. With this extended definition we have
\begin{multline}
\label{eq:exp2}
   \bh' = \sum_{0<\Re\lambda<1+\lambda_\bc} \bm{\Phi}^\lambda + \bh'_0 \\[-1ex] 
   \mbox{with}\quad
   \bm{\Phi}^\lambda\in\bm{Z}_M^\lambda \ \ \mbox{and}\ \ 
   \bh'_0\in\rP\rH^{2+\lambda_\bc-\eta}(\cB(\bc,r_0),\cP')^2\ \ \forall\eta>0.
\end{multline}
It remains to find for which values of $\lambda$ the space $\bm{Z}_M^\lambda$ is not reduced to $\{0\}$. The sole integer in the strip $0<\Re\lambda<1+\lambda_\bc$ is $\lambda=1$. For any other value of $\lambda$, we prove as in the proof of Theorem \ref{th:7} that $\bm{Z}_M^\lambda=\{0\}$.

Let $\lambda=1$ and suppose $\bm{\Phi} \in \bm{Z}_M^1$. As in the
proof of Theorem \ref{th:7} define $\Psi = \curl \bm{\Phi}$ and $\Pi
= \Div \bm{\Phi}$.  Since $\bm{P}^{\lambda-2} = \{0\}$, 
we deduce as above that
\[
   \tsfrac{1}{n^2}\Psi\in S_\Delta^{0} \quad\mbox{and}\quad
   \Pi \in S_\Delta^{0},
\]
where $S_\Delta^{0}$ is defined in \eqref{eq:SDelta}. Since the space
$S_\Delta^{0}$ is generated by $1$ and $\log r$ (see Lemma
\ref{lem:SDelta}), there are constants
$\gamma_0,\gamma_1,\tilde\gamma_0,\tilde\gamma_1 \in \bbR$, such that
\beq
\label{eq:c1}
	\curl \bm{\Phi} = n^2 (\gamma_0 +\gamma_1\log r) \quad \mbox{and} \quad
	\Div \bm{\Phi} = \tilde\gamma_0 + \tilde\gamma_1\log r.
\eeq
These conditions are necessary for $M\bm{\Phi}=0$ to hold. Calculating $M\bm{\Phi}$ with the Ansatz \eqref{eq:c1}, we find
\[
   M\bm{\Phi} = \gamma_1 \Curl \log r - \tilde\gamma_1\grad \log r.
\]
The equation $M\bm{\Phi}=0$ implies that $\gamma_1=\tilde\gamma_1=0$. Therefore we are left with
\begin{equation}
\label{eq:Phi}
	\curl \bm{\Phi} = n^2 \gamma_0  \quad \mbox{and} \quad
	\Div \bm{\Phi} = \tilde\gamma_0.
\end{equation}
It remains to find all solutions to \eqref{eq:Phi} in $\bm{S}^1_{[1]} / \bm{P}^1$.
Equivalently, we can look for:
\begin{subequations}
\begin{align}
	\label{reg8a} \mbox{all solutions of} && \curl \bm{\Phi} &= 0 &\mbox{and}&& \Div \bm{\Phi} &= 0, \\[-0.5ex]
	\label{reg8b} \mbox{a particular solution of} && \curl \bm{\Phi} &= 0 &\mbox{and}&& \Div \bm{\Phi} &= 1, \\[-0.5ex]
	\label{reg8c} \mbox{a particular solution of} && \curl \bm{\Phi} &= n^2 &\mbox{and}&& \Div \bm{\Phi} &= 0.
\end{align}
\end{subequations}

Case \eqref{reg8a}.  Suppose that $\bm{\Phi} \in \bm{S}^1_{[1]} / \bm{P}^1$
satisfies \eqref{reg8a}.  Since $\curl \bm{\Phi} = 0$ there exists a
potential $V \in \bm{S}^2_{[2]}$ such that $\bm{\Phi} = \grad V$.  Then
$\Delta V = 0$, so $V$ belongs to $S_\Delta^2$.  Hence by  Lemma
\ref{lem:SDelta}, $V$ must be a homogeneous polynomial of degree
$2$. This implies that $\bm{\Phi}$ is a homogeneous polynomial of degree
$1$, and so $\bm{\Phi} = 0$ in $\bm{Z}_M^1$.

Case \eqref{reg8b}.  Suppose that $\bm{\Phi} \in \bm{S}^1_{[1]} / \bm{P}^1$
satisfies \eqref{reg8b}.  Again there exists a potential $V \in
\bm{S}^2_{[2]}$ such that $\bm{\Phi} = \grad V$ and $\Delta V = 1$.  A
particular solution is $V = \tsfrac{1}{2} x^2$, hence $\bm{\Phi}$ is a
polynomial of degree $1$, and so $\bm{\Phi} = 0$ in $\bm{Z}_M^1$ again.

Case \eqref{reg8c}.  Suppose that $\bm{\Phi} \in \bm{S}^1_{[1]} / \bm{P}^1$ satisfies \eqref{reg8c}.  Since $\Div\bm{\Phi}=0$, there exists $W \in \bm{S}^2_{[2]}$ such that $\bm{\Phi} = \Curl W$.  Then $\Delta W = n^2$.  Using \cite[Proposition 4.1 and \S4.3] {ddkppp2013} an explicit solution can be found of the form
$$
	W(r,\theta) = r^2 w_0(\theta) + r^2 \log r \; w_1(\theta),
$$
where $w_0$ and $w_1$ are smooth functions on $[0,\omega_\bc]$ and
$[\omega_\bc,2\pi]$. Calculating $\bm{\Phi} = \Curl W$ we obtain an
expression for $\bm{\Phi}$ in the form $\bm{\Phi}  = r
\bm\phi_{0}(\theta) + r \log r \bm\phi_{1}(\theta)$. This proves in
the end that $\bm{Z}_M^1$ has dimension $1$ and, combined with
expansion \eqref{eq:exp2}, achieves the proof of Theorem
\ref{th:hx+}. 
\end{proof}

\begin{remark}
So for the case of simple interface edges between two materials the
regularity provided by Theorem \ref{th:7} is quasi-optimal, since the solution asymptotics contains a singular function with a $r\log r$ term, thus not in $\rP\rH^2$ (but still in $\rP\rH^{2-\eta}$ for any positive $\eta$).
The next singularity in the expansion of $\bh'$ could be determined,
too. It originates from the first singularity of $h_z$ as described in
Theorem \ref{th:elschner}, via the right hand side $\boldf$ of
equation \eqref{eq:bh'}, cf.\ \eqref{eq:bf}. This next singularity has the form 
\[
   \tilde\gamma_\bc\, r^{1+\lambda_\bc} \tilde{\bm\phi}(\theta)
\]
leading to a new regular part belonging to $\rP\rH^{3-\eta}(\cB(\bc,r_0),\cP')^2$ for all $\eta>0$.
\end{remark}

\section{Conclusions and consequences for numerical methods} 
\label{sec fem}

The almost optimal piecewise regularity results in the previous two
sections for the transverse components of the magnetic field in
translationally invariant, periodic media are somewhat unexpected. 
When one is used to the scalar transmission problem $\Delta_\epsilon=
\nabla \cdot(\epsilon\nabla \cdot)$ or studies the general Maxwell
regularity theory in \cite{costabel,costabel00}, one would expect in
general only piecewise $\rH^{1+s}$ regularity for some possibly 
``small'' $s > 0$ (as outlined above). This is particularly pronounced 
in the case of a cross point, where four regions $Q'_\ell$ with 
$\epsilon_1 = \epsilon_3 \ll \epsilon_2 = \epsilon_4$ meet in a point. 
In that case the solution to the scalar transmission problem is only 
in $\rP\rH^{1+s}$ with $s<\sigma_\varepsilon \ll 1$ (cf. 
\cite[Th.8.1]{costabel}), whereas the transverse components of the
magnetic field $\bh'$ are in $\rP\rH^{2-\eta}$, for all $\eta > 0$. 
In the case of PCFs, where we only have simple interfaces between two
materials, there is also no significant loss of regularity near 
reentrant corners in any of the subregions $Q'_\ell$, as in the scalar
elliptic transmission problem. 

The results also carry over to piecewise $C^2$ cross sections 
without cusps. This follows immediately from the above when the 
interfaces are straight near the corners, but the analysis can 
also be extended to the general case where there exists a smooth, 
local diffeomorphism that straightens the interfaces abutting to 
the same corner. We do not give any details here but refer to 
\cite[\S2-3]{kondratiev} instead. 

The improved regularity results are of interest in the design and
analysis of more efficient finite element methods for PCFs.  
The convergence of finite element methods depends only on the 
piecewise regularity of the solution 
(cf.~\cite{brambleking,chenzou,chenduzou}), and so the results in 
this paper suggest the following conclusions:
\begin{enumerate}
\item Since the regularity of $\bH$ is better than $\bE$ when 
$\mu = \mu_0$ is constant (Corollary \ref{cor:1th:3}), it is 
better to apply a numerical method to a formulation of Maxwell \
equations based on \eqref{maxwell12} instead of \eqref{maxwell012}. 
It is also definitely better to work with the transverse components 
of the magnetic field rather than with the transverse components of 
the electric field. The results in Section \ref{sec:reg++} (see also
\cite{costabel}) imply that $e_x$ and $e_y$ have significantly lower
regularity and are only in $\rP\rH^{\sigma_\epsilon-\eta}$, for any 
$\eta>0$. (Recall that $\sigma_\epsilon < 1$.) Most papers in the
numerical modelling of PCFs do in fact choose the magnetic field, 
but some of them state that one could equally choose the electric field 
(e.g.~\cite{saitohkoshiba}). Our results imply that this would lead 
to a significantly worse convergence (at least for materials with 
constant magnetic permeability).\vspace{1ex}
\item
The reduction in \S\ref{sec:pertrans} to a family of two-dimensional
eigenproblems with bilinear forms $a\xib(\bu,\bv)$ provides an 
obvious advantage for computations over the original three-dimensional
eigenproblem in \S\ref{sec:periodic}. Moreover, our regularity theory
shows that $\bh'$ has better regularity than $h_z$ (see Thoerem 
\ref{th:7} vs. Theorem \ref{th:elschner}) and this suggests that 
it may be advantageous (in terms of convergence rates) to also 
eliminate $h_z$ from the reduced eigenproblem in \S\ref{sec:pertrans} 
and to solve a problem where $\bh'$ is the unknown eigenfunction 
(as in \cite{norton2012apnum}). Even without eliminating $h_z$ it 
may be possible to develop bespoke convergence results that show 
quasi-optimal convergence for $\bh'$ given its improved regularity. 
For PCFs this also seems to be a better approach than that advocated 
in \cite{elschner} which reduces \eqref{maxwell012} to a problem in 
$h_z$ and $e_z$ only.\vspace{1ex}
\item  
The convergence rate of spectral methods, 
such as the planewave expansion me\-thod 
(cf.\cite{pottage2003robust,joannopoulosbook,norton2012apnum}), 
depends on the global regularity of the solution 
(see \cite{norton2010,norton2012apnum}), whereas 
finite element methods (with their local basis 
functions) are able to exploit any piecewise 
regularity (cf.~\cite{brambleking,chenzou,chenduzou}).
Globally the regularity of $\bh'$ is restricted to 
$\rH^{3/2-\eta}$ (cf.~Corollary \ref{cor:66}) while 
locally within each material $\bh'$ is $\rH^{2-\eta}$ 
(cf.~Theorem \ref{th:7}). Therefore, we would expect 
that finite element methods applied to a reduced 
eigenproblem in $\bh'$ only, such as described in 
\cite{norton2012apnum}, will converge faster with 
respect to the number of degrees of freedom than 
the planewave methods that were employed there. 
With an adapted mesh the performance of the finite 
element method will be even better.
\end{enumerate}

\paragraph{Acknowledgements}
We would like to thank Johannes Elschner for 
his useful advice on this problem.

\def\cprime{$'$}

\end{document}